\documentclass[a4paper,reqno]{amsart}

\usepackage[foot]{amsaddr}
\usepackage{amsmath}
\usepackage{amssymb}
\usepackage{amsthm}
\usepackage{hyperref}
\usepackage[english]{isodate}
\usepackage{color}

\theoremstyle{plain}
\newtheorem{conjecture}{Conjecture}
\newtheorem{lemma}{Lemma}
\newtheorem{theorem}{Theorem}

\theoremstyle{definition}

\newtheorem{question}{Question}

\newtheorem{remark}{Remark}

\DeclareMathOperator{\ex}{ex}

\let\geq\geqslant
\let\leq\leqslant
\let\ge\geqslant

\title{A lower bound on the zero forcing number}

\author{Randy Davila$^{1}$}
\address{$^1$Department of Mathematics and Statistics, University of Houston--Downtown, Houston, TX 77002, USA}

\author{Thomas Kalinowski$^{2,3}$ and Sudeep Stephen$^{3,4}$}
\address{$^2$School of Science and Technology, University of New England, Armidale, 2351 NSW, Australia}
\address{$^3$School of Mathematical and Physical Sciences, University of Newcastle, 2308 NSW, Australia}
\address{$^4$School of Mathematical Sciences, National Institute of Science Education and Research, Bhubaneswar, India}

\email{davilar@uhd.edu,tkalinow@une.edu.au,sudeep.stephen@niser.ac.in}

\date{\today}

\begin{document}

\begin{abstract}
  In this note, we study a dynamic vertex coloring for a graph $G$. In particular, one starts with a
  certain set of vertices black, and all other vertices white. Then, at each time step, a black
  vertex with exactly one white neighbor forces its white neighbor to become black. The initial set
  of black vertices is called a \emph{zero forcing set} if by iterating this process, all of the
  vertices in $G$ become black. The \emph{zero forcing number} of $G$ is the minimum cardinality of
  a zero forcing set in $G$, and is denoted by $Z(G)$. Davila and Kenter have conjectured in 2015
  that $Z(G)\geq (g-3)(\delta-2)+\delta$ where $g$ and $\delta$ denote the girth and the minimum
  degree of $G$, respectively. This conjecture has been proven for graphs with girth $g \leq 10$. In
  this note, we present a proof for $g \geq 5$, $\delta \geq 2$, thereby settling the conjecture.
\end{abstract}

\maketitle

\section{Introduction}\label{sec:intro}
For a two-coloring of the vertex set of a simple graph $G=(V,E)$ consider the following color-change
rule: a white vertex $u$ is converted to black if it is the only white neighbor of some black vertex
$v$. We call such a black vertex $v$ a \emph{forcing vertex} and say $v$ forces $u$.  Given a
two-coloring of $G$, the \emph{derived set} is the set of black vertices obtained by applying the
color-change rule until no more changes are possible. A \emph{zero forcing set} for $G$ is a subset
of vertices $S\subseteq V$ such that if initially the vertices in $S$ are colored black and the
remaining vertices are colored white, then the derived set is the complete vertex set $V$.  The
minimum cardinality of a zero forcing set for the graph $G$ is called the \emph{zero forcing number}
of $G$, denoted by $Z(G)$. This concept was introduced by the AIM Minimum Rank – Special Graphs Work
Group~\cite{AIM-2008-Zeroforcingsets} as a tool to bound the minimum rank of matrices associated
with the graph $G$. Since its introduction the zero-forcing number has been studied as an
interesting graph invariant with various
applications~\cite{BarioliBarrettFallatEtAl2010,BarioliBarrettFallatHallHogbenShaderDriesscheHolst-2012-ParametersRelatedto,EdholmHogbenHuynhEtAl2012,Genter1,Gentner2018,Lu.etal_2015_Proofconjecturezero}.
Moreover, it has been established that the zero forcing problem is $NP$-complete~\cite{A08}, which
motivates the search for easily computable bounds for $Z(G)$. The following conjecture was made
by Davila and Kenter~\cite{DavilaKenter}.
\begin{conjecture}\label{girthConj}
  If $G$ is a graph with girth $g\ge 3$ and minimum degree $\delta \geq 2$, then $Z(G) \ge \delta + (\delta - 2)(g - 3)$.    
\end{conjecture}
Genter et al.~\cite{Genter1}, Genter and Rautenbach~\cite{Gentner2018} and~Davila and
Henning~\cite{Davila2017} have shown that the statement is true for $g \leq 10$. In this note we
provide a complete proof for $g \geq 5$.
\begin{theorem}\label{thm:result}
  Let $G$ be a graph with girth $g \geq 5$ and minimum degree $\delta \geq
  2$. Then $Z(G) \ge \delta + (\delta - 2)(g - 3)$.    
\end{theorem}
\begin{remark}
  A proof of the slightly weaker bound $Z(G)\geq 2 + (\delta - 2)(g - 3)$ has recently been found by
  F\"urst and Rautenbach~\cite{Furst2017}. For large values of $g$ and $\delta$, a stronger bound
  has been proved by Kalinowski et al.~\cite{Kalinowski2017}.
\end{remark}

For a positive integer $n$, we use $[n]$ to denote the set $\{1,2,\dots,n\}$, and in addition we set
$[0]=\emptyset$. We restrict ourselves to undirected finite simple graphs, and use the following
notation, referring the reader to any graph theory textbook such as~\cite{Diestel2017} for more
details. Let $G = (V,E)$ be a graph. Two vertices $v,w\in V$ are called \emph{neighbors}, or
\emph{adjacent} vertices, whenever $\{v,w\}\in E$. The \emph{neighborhood} of $v\in V$ is the set of
neighbors of~$v$, denoted by $N(v) = N_G(v)$. The \emph{degree} of $v\in V$ is the cardinality of
its neighborhood, and is denoted by $\deg_G(v) = \lvert N(v)\rvert$. The minimum vertex degree in
$G$ is denoted by $\delta(G)$. A cycle of length $\ell$ is denoted by $ C_\ell$. The girth of $G$,
denoted $g = g(G)$, is the length of a shortest cycle in $G$.

Our strategy for proving Theorem~\ref{thm:result} is the same that was used in~\cite{Davila2017}. We
assume that $G$ has a zero forcing set $S$ of size $\lvert S\rvert<\delta + (\delta - 2)(g - 3)$,
use the minimum degree condition to bound the number of edges in the subgraph of $G$ induced by the
first $g-2$ forcing vertices together with the union of their neighborhoods, and finally show that
this implies the existence of a short cycle, contradicting the assumption on the girth of $G$. For
this argument we need bounds on the \emph{extremal number} $\ex(n; \{C_3, C_4, \dots, C_\ell\})$,
which is defined as the maximum number of edges in a simple graph with $n$ vertices and girth at
least $\ell+1$. For $\ell=3$ this number is given by Mantel's theorem~\cite{Mantel1907} from 1907,
which is a special case of Tur\'an's theorem~\cite{Turan1941} (Proofs for this classic result which
started the area of extremal graph theory can be found in most graph theory textbooks, for
instance~\cite[Chapter 7]{Diestel2017}).
\begin{theorem}[Mantel~\cite{Mantel1907}]\label{thm:mantel}
  For every positive integer $n$, $\ex\left(n;\{C_3\}\right)=\lfloor n^2/4\rfloor$.
\end{theorem}
For $\ell\geq 4$ we will use the follwing more recent result by Abajo and Di\'anez~\cite{AbajoDian2015}.
\begin{theorem}[Theorem 1 in \cite{AbajoDian2015}]\label{thm:extremal_function}
  Let $\ell\geq 4$ and $\ell+1\leq n\leq 2\ell$ be integers. Then
  \[ \ex\left(n; \{C_3, C_4, \dots, C_\ell\}\right)=
    \begin{cases}
      n & \text{if }\ell+1\leq n\leq\lfloor 3\ell/2\rfloor,\\
      n+1 & \text{if } \lfloor 3\ell/2\rfloor+1\leq n\leq 2\ell-1,\\
      n+2 & \text{if }n=2\ell.
    \end{cases}
  \]
\end{theorem}

\section{Proof of Theorem~\ref{thm:result}}\label{sec:proof}
In order to prove Theorem~\ref{thm:result}, suppose $G=(V,E)$ is a graph with girth $g\geq 5$,
minimum degree $\delta\geq 2$, and that $S\subseteq V$ is a zero forcing set with
$\lvert S\rvert\leq(\delta-2)(g-3)+\delta-1$. Let $x_1,\dots,x_t$ be a chronological list of forcing
vertices resulting in all of $V$ becoming black starting with $S$ as the initial set of black
vertices, and let $y_i$ be the vertex that is forced by $x_i$. Let $C$ be a cycle of length $g$ in
$G$. Every vertex in $C$ has at least $\delta-2$ neighbors outside $C$, and from $g\geq 5$ it
follows that no two vertices $u$ and $v$ in $C$ have a common neighbor outside $C$, because
otherwise the shorter path between $u$ and $v$ on $C$, together with the edges joining $u$ and $v$ to their common
neighbor outside $C$ give a cycle of length at most $\lfloor g/2\rfloor+2<g$. As a consequence,
$\lvert V\rvert\geq g+g(\delta -2)=g(\delta-1)$, and therefore
\[t=\lvert V\setminus S\rvert = \lvert V\rvert - \lvert S\rvert \ge g(\delta - 1) - (\delta -2)(g -3) - \delta + 1 = g + 2\delta-5 \ge g - 1.\]
In particular $t\geq g-2$, and this allows us to define the set
$X=\{x_1,x_2,\dotsc,x_{g-2}\}$. Modifying the notation from~\cite{Davila2017}, we define
the sets $S_1=S\cap N(x_1)$ and
\[S_i=(S\cap N(x_i))\setminus\bigcup_{j=1}^{i-1}N(x_j)=(S\cap
  N(x_i))\setminus\bigcup_{j=1}^{i-1}S_j\] for $i=2,3,\dotsc,g-2$. Equivalently, $S_i$ is the set
of initially black neighbors of $x_i$ which are not adjacent to any $x_j$ for $j<i$. In particular,
the sets $S_i$ are pairwise disjoint. We also define the sets
\begin{align*}
  S_X^* &= \bigcup_{i=1}^{g-2}S_i=S\cap\bigcup_{i=1}^{g-2}N(x_i), & S_X &= (S\cap X)\setminus S_X^*. 
\end{align*}
In words, $S_X^*$ is the set of vertices $v\in S$ which are adjacent to at least one of the vertices
in $X$, and $S_X$ is the set of vertices in $X$ that are initially black and not adjacent to
any vertex in $X$. In particular, the sets $S^*_X$ and $S_X$ are disjoint subsets of $S$,
hence $\lvert S\rvert\geq\lvert S^*_X\rvert+\lvert S_X\rvert$.

We define two auxiliary graphs $H_1=(X,E_1)$ and $H_2=(X,E_2)$, both on the vertex set $X$. The
graph $H_1$ is the subgraph of $G$ induced by $X$, also denoted by $G[X]$, and two vertices
$x_j,x_i\in X$ with $j<i$ are adjacent in $H_2$ if and only if (1) they have a common neighbor $u$
in $G$, and (2) this common neighbor is not adjacent (in $G$) to any $x_k$ with $k<j$. More
formally, the edge sets of $H_1$ and $H_2$ are given by
\begin{align*}
  E_1 &= \left\{\,\{x_j,x_i\}\,:\,1\leq j<i\leq g-2,\,\{x_j,x_i\}\in E\right\},\\
  E_2 &= \left\{\,\{x_j,x_i\}\,:\,1\leq j<i\leq g-2,\,N(x_i)\cap(S_j\cup\{y_j\})\neq\emptyset\right\}.
\end{align*}
The graph $H_1$ is a forest, because it is a subgraph with less than $g$ vertices of the graph $G$
which has girth $g$. We remark that $S_X$ is precisely the set of isolated vertices in the graph
$H_1$: If $x_i$ is not isolated in $H_1$ then it is adjacent to some $x_j$, and therefore in
$S^*_X$, and if $x_i$ is isolated in $H_1$ then it cannot be forced by any of the vertices $x_j$
with $j<i$, and therefore it must have been black in the beginning which gives $x_i\in S_X$. The
assumption $g\geq 5$ implies that two vertices which are adjacent in $G$ do not have a common
neighbor in $G$, and two vertices which are non-adjacent in $G$ have at most one common neighbor in
$G$. In particular, $E_1\cap E_2=\emptyset$ and $\lvert N(x_j)\cap N(x_i)\rvert=1$ for every
$\{x_j,x_i\}\in E_2$.
\begin{lemma}\label{lem:bound_Si}
  For every $i\in[g-2]$,
  $\left\lvert S_i\right\rvert=\deg_G(x_i)-1-\left\lvert\left\{\,j\in[i-1]\,:\,\{x_j,x_i\}\in
      E_2\right\}\right\rvert$.
\end{lemma}
\begin{proof}
  The vertex $x_1$ forces in the first step. So all but one of its neighbors are initially black,
  and this implies $\lvert S_1\rvert=\deg_G(x_1)-1$. For $i\geq 2$, since $x_i$ forces in step $i$,
  all but one of its neighbours are black after the first $i-1$ forcing steps:
  $\left\lvert N(x_i)\cap\left( S\cup\{y_1,\dots,y_{i-1}\}\right)\right\rvert=\deg_G(x_i)-1$.  The
  sets $S_j\cup\{y_j\}$ for $j\in[i-1]$ are pairwise disjoint, and we obtain a partition
\[N(x_i)\cap\left(
    S\cup\{y_1,\dots,y_{i-1}\}\right)=S_i\cup\bigcup_{j=1}^{i-1}\left(\,N(x_i)\cap\left(S_j\cup\{y_j\}\right)\,\right).\]
In particular, for every vertex
$v\in\bigcup_{j=1}^{i-1}\left(\,N(x_i)\cap\left(S_j\cup\{y_j\}\right)\,\right)$, there is a unique
index $j\in[i-1]$ with $v\in S_j\cup\{y_j\}$, and then $\{x_j,x_i\}\in E_2$ by definition of
$E_2$. Conversely, for every $j\in[i-1]$ with $\{x_j,x_i\}\in E_2$, there is a corresponding vertex
$v\in N(x_i)\cap\left(S_j\cup\{y_j\}\right)$, and this establishes a bijection between the sets
$\bigcup_{j=1}^{i-1}\left(\,N(x_i)\cap\left(S_j\cup\{y_j\}\right)\,\right)$ and
$\left\{\,j\in[i-1]\,:\,\{x_j,x_i\}\in E_2\right\}$. Consequently,
\begin{multline*}
  \lvert S_i\rvert = \left\lvert N(x_i)\cap\left( S\cup\{y_1,\dots,y_{i-1}\}\right)\right\rvert -
  \left\lvert\bigcup_{j=1}^{i-1}\left(\,N(x_i)\cap\left(S_j\cup\{y_j\}\right)\,\right)\right\rvert\\
  =\deg_G(x_i)-1-\left\lvert\left\{\,j\in[i-1]\,:\,\{x_j,x_i\}\in E_2\right\}\right\rvert.\qedhere
\end{multline*}
\end{proof}
\begin{lemma}\label{lem:bound_S*}
  $\displaystyle \lvert S^*_X\rvert \geq (g-2)(\delta-1)-\lvert E_2\rvert$.
\end{lemma}
\begin{proof}
  This is obtained by summing the equalities from Lemma~\ref{lem:bound_Si} over $i\in[g-2]$, taking into
  account that the sets $S_i$ are pairwise disjoint, and that $\deg_G(x_i)\geq\delta$ for all $i$:
  \begin{multline*}
    \left\lvert S^*_X\right\rvert=\left\lvert S_1\cup S_2\cup\dots\cup
      S_{g-2}\right\rvert=\sum_{i=1}^{g-2}\left\lvert S_i\right\rvert
    = \sum_{i=1}^{g-2}\left(\deg_G(x_i)-1-\left\lvert\left\{\,j\in[i-1]\,:\,\{x_j,x_i\}\in E_2\right\}\right\rvert\right)\\
    \geq (g-2)(\delta-1) - \sum_{i=1}^{g-2}\left\lvert\left\{\,j\in[i-1]\,:\,\{x_j,x_i\}\in
        E_2\right\}\right\rvert = (g-2)(\delta-1)-\lvert E_2\rvert.\qedhere
  \end{multline*}
\end{proof}
Combining Lemma~\ref{lem:bound_S*} with our assumption on $\lvert S\rvert$ we obtain
\[(\delta-2)(g-3)+\delta-1\geq\lvert S\rvert\geq\lvert S^*_X\rvert+\lvert
  S_X\rvert\geq(g-2)(\delta-1)-\lvert E_2\rvert + \lvert
  S_X\rvert,\]
and after rearranging,
\begin{equation}\label{eq:E2_lower_bound}
  \lvert E_2\rvert \geq (g-2)(\delta-1)+ \lvert
  S_X\rvert-(\delta-2)(g-3)-\delta+1=g-3+\lvert S_X\rvert.
\end{equation}
Let $X_1,\dots,X_k$ be the vertex sets of the connected components of $H_1$ such that
$\lvert X_1\rvert\geq\lvert X_2\rvert\geq\dots\geq\lvert X_k\rvert$, and put $l=k-\lvert
S_X\rvert$. Equivalently, $l=0$ if $E_1=\emptyset$ and $l=\max\{i\,:\,\lvert X_i\rvert\geq 2\}$ if
$E_1\neq\emptyset$.  Let $E_s(X_p)$ for $s\in\{1,2\}$ and $p\in[l]$ be the set of edges
$\{x_j,x_i\}\in E_s$ with both vertices in $X_p$, and let $E_2(X_p,X_q)$ for $1\leq p<q\leq k$ be
the set of edges $\{x_j,x_i\}\in E_2$ with one vertex in $X_p$ and the other in $X_q$. This provides
a partition $E_2=E'_2\cup E''_2$ where
\begin{align*}
  E'_2 &= \bigcup_{1\leq p\leq l} E_2(X_p), & E''_2 &= \bigcup_{1\leq p<q\leq k} E_2(X_p,X_q).
\end{align*}
\begin{lemma}\label{lem:star}
  Let $u\in V$, and supposed that $N=N_G(u)\cap X$ is non-empty. Let $j=\min\{i\,:\,x_i\in
  N\}$. Then the subgraph of $H_2$ induced by $N$ is a star with center $x_j$, that is,
\[\{\,\{v,w\}\in E_2\,:\,v,w\in N\,\}=\{\,\{x_j,v\}\,:\,v\in
    N\setminus\{x_j\}\,\}.\] 
\end{lemma}
\begin{proof}
  Vertex $x_j$ is forcing in step $j$, and its neighbor $u$ is not adjacent (in $G$) to any vertex
  $x_i$ with $i<j$. This implies that either $u$ was initially black, that is, $u\in S$, or $u$ is
  the vertex forced by $x_j$, that is, $u=y_j$. In any case $u\in S_j\cap\{y_j\}$. For every $v\in
    N\setminus\{x_j\}$, we have $v=x_i$ for some $i>j$, and then $u\in N(x_i)\cap(S_j\cap\{y_j\})$
    implies $\{x_j,x_i\}\in E_2$. Now fix two vertices $v,w\in N\setminus\{x_j\}$, say $v=x_i$ and
    $w=x_{i'}$ with $j<i<i'$. Since the unique common neighbor (in $G$) of $v$ and $w$ is in
    $S_j\cup\{y_j\}$ which is disjoint from $S_{i}\cup\{y_i\}$, we conclude $\{v,w\}\not\in E_2$.
\end{proof}
\begin{lemma}\label{lem:bound_E2_p}
  For every $p\in[l]$, $\lvert E_2(X_p)\rvert=\lvert X_p\rvert-2$.
\end{lemma}
\begin{proof}
  Let $\{v,w\}\in E_2(X_p)$, and let $u$ be the unique common neighbor (in $G$) of $v$ and $w$. Then
  $u\in X_p$, because otherwise the path between $v$ and $w$ in the tree $(X_p,E_1(X_p))$ together
  with the edges $\{u,v\}$ and $\{u,w\}$ forms a cycle in $G$ of length at most
  $\lvert X_p\rvert+1\leq g-1$. Fix some $u\in X_p$ and consider the set
  $N=N_G(u)\cap X\subseteq X_p$. Let $j$ be the smallest index with $x_j\in N$. By Lemma~\ref{lem:star},
\[\left\lvert\{\,\{v,w\}\in E_2\,:\,v,w\in N\,\}\right\rvert=\left\lvert\{\,\{x_j,v\}\,:\,v\in
    N\setminus\{x_j\}\,\}\right\rvert = \deg_{H_1}(u)-1.\]
Since for every $\{v,w\}\in E_2(X_p)$ there is a unique $u\in X_p$ with $v,w\in N_G(u)$, we conclude
that
\[\lvert E_2(X_p)\rvert = \sum_{u\in X_p}\left(\deg_{H_1}(u)-1\right)=2\lvert E_1(X_p)\rvert-\lvert
  X_p\rvert=2(\lvert X_p\rvert-1)-\lvert X_p\rvert=\lvert X_p\rvert-2.\qedhere\]
\end{proof}
As a consequence of Lemma~\ref{lem:bound_E2_p}, 
\[\lvert E'_2\rvert=\sum_{p=1}^l\left\lvert E_2(X_p)\right\rvert = \sum_{p=1}^l\left(\left\lvert X_p\right\rvert-2\right) =g-2-\lvert S_X\rvert-2l.\]
Combining this with~\eqref{eq:E2_lower_bound},
\begin{equation}\label{eq:E2''_lower_bound}
  \lvert E''_2\rvert = \lvert E_2\rvert-\lvert E'_2\rvert\geq \left(g-3+\lvert S_X\rvert\right)-\left(g-2-\lvert
    S_X\rvert-2l\right)=2\left(\lvert S_X\rvert+l\right)-1=2k-1.
\end{equation}
Next we consider the graph $H_3=(X,E_1\cup E''_2)$, that is, $H_3$ is obtained from $H_1$ by adding
the edges of $H_2$ which connect distinct components of $H_1$. The next lemma shows that short
cycles in $H_3$ can be lifted to short cycles in $G$.
\begin{lemma}\label{lem:inducing_cycles}
  Let $C$ be a cycle of length $\lambda$ in $H_3$ which contains $s$ edges from $E''_2$. Then there
  exists a cycle $C'$ of length at most $\lambda+s$ in $G$. 
\end{lemma}
\begin{proof}
  Fix any $u\in V$ with $N=N_G(u)\cap X\neq\emptyset$, and let $j$ be the smallest index with
  $x_j\in N$. By Lemma~\ref{lem:star},
  \[\{\,\{v,w\}\in E_2\,:\,v,w\in N\,\} = \{\,\{x_j,v\}\,:\,v\in N\setminus\{x_j\}\,\}.\]
  Let's call this set $E_2(u)$. It follows that for every vertex $u$, either $E_2(u)\cap
  C=\emptyset$, or $\lvert E_2(u)\cap C\rvert=1$, or $E_2(u)\cap C$ consists of two adjacent
  edges $\{v,v'\}$ and $\{v',v''\}$. We obtain the required cycle $C'$ by starting with $C$ and
  doing the following replacements:
  \begin{itemize}
  \item For every $u$ with $\lvert E_2(u)\cap C\rvert=1$, say $E_2(u)\cap C=\{\,\{v,w\}\,\}$, replace
    $\{v,w\}$ by $\{u,v\}$ and $\{u,w\}$.
  \item For every $u$ with $\lvert E_2(u)\cap C\rvert=2$, say $E_2(u)\cap C=\{\,\{v,v'\},\,\{v',v''\}\,\}$, replace
    $\{v,v'\}$ and $\{v',v''\}$ by $\{u,v\}$ and $\{u,v''\}$.\qedhere
  \end{itemize}
\end{proof}
\begin{lemma}\label{lem:bound_E2_edges_in_cycles}
  Every cycle in the graph $H_3$ has at least $\lceil(k+2)/2\rceil$ edges in $E''_2$.
\end{lemma}
\begin{proof}
  Since $H_1$ is a forest with $g-2$ vertices and $k$ connected components, we have $\lvert E_1\rvert= g-2-k$. Let $C$
  be a cycle in $H_3$ which has $s$ edges in $E''_2$, and let $\lambda$ be the length of $C$. Then
  $\lambda\leq s+\lvert E_1\rvert=s+(g-2-k)$, and by Lemma~\ref{lem:inducing_cycles}, $G$ contains a
  cycle of length at most $2s+(g-2-k)$. Since $G$ has girth $g$, this implies $2s+(g-2-k)\geq g$, hence $2s\geq k+2$.
\end{proof}
\begin{lemma}\label{lem:bound_E2_pq_strong}
 For every pair $(p,q)$ with $1\leq p< q\leq k$, $\lvert E''_2(X_p,X_q)\rvert\leq 1$. 
\end{lemma}
\begin{proof}
  Suppose $\lvert E''_2(X_p,X_q)\rvert\geq 2$. Then $H_3$ contains a cycle with $2$ edges from
  $E''_2$, and $2\geq(k+2)/2$ by Lemma~\ref{lem:bound_E2_edges_in_cycles}, which implies $k\leq 2$, hence
  $k=2$. Then~(\ref{eq:E2''_lower_bound}) implies $\lvert E''(X_1,X_2)\rvert=\lvert E''\rvert\geq
  3$, and consequently, there is a cycle $C$ in $H_3$ which has $s=2$ edges in $E''_2$ and does not
  use all the edges in $\lvert E_1\rvert$, so its length is 
\[\lambda\leq 2+\lvert E_1\rvert-1=\lvert E_1\rvert+1=(g-2-2)+1=g-3.\] 
Then Lemma~\ref{lem:inducing_cycles} implies that $G$ contains a cycle of length at most
$\lambda+2\leq g-1$ edges, which contradicts the assumption on the girth of $G$. 
\end{proof}
\begin{lemma}\label{lem:k_upper_bound}
  $k\in\{5,6\}$.
\end{lemma}
\begin{proof}
  For this proof, we introduce another graph $H_4=([k],E_4)$ with $\{i,j\}\in E_4$ if and only if
  $E''_2$ contains an edge between $X_i$ and $X_j$. By Lemma~\ref{lem:bound_E2_pq_strong} there is a
  one-to-one correspondence between $E''_2$ and $E_4$. It follows that
  $\binom{k}{2}\geq\lvert E_4\rvert=\lvert E''_2\rvert\geq 2k-1$, hence $k\geq 5$.  Suppose
  $k\geq 7$. Using~(\ref{eq:E2''_lower_bound}) and
  Theorem~\ref{thm:extremal_function} for $\ell=\lfloor(k+1)/2\rfloor$, we have
  $\lvert E_4\rvert=\lvert E''_2\rvert\geq
  2k-1>k+2\geq\ex\left(k,\{C_3,\dots,C_{\lfloor(k+1)/2\rfloor}\}\right)$. As a consequence, $H_4$
  contains a cycle of length $s\leq \lfloor(k+1)/2\rfloor<\lceil(k+2)/2\rceil$, and this corresponds
  to a cycle in $H_3$ which contains at most $s$ edges from $E''_2$, which is impossible by
  Lemma~\ref{lem:bound_E2_edges_in_cycles}.
\end{proof}
\begin{proof}[Proof of Theorem~\ref{thm:result}]
  Let $H_4=([k],E_4)$ be the graph introduced in the proof of Lemma~\ref{lem:k_upper_bound}, and
  note that $k\in\{5,6\}$ by Lemma~\ref{lem:k_upper_bound}. For both possible values for $k$,
  $\lvert E_4\rvert=\lvert E''_2\rvert\geq 2k-1>k^2/4$, and by Theorem~\ref{thm:mantel} this implies that
  $H_4$ contains a triangle. As a consequence, $H_3$ contains a cycle with three edges from
  $E''_2$. The length of this cycle is at most $3+\lvert E_1\rvert=3+(g-2-k)$, and by
  Lemma~\ref{lem:inducing_cycles} $G$ contains a cycle of length at most $6+(g-2-k)\leq g-1$, which
  is the required contradiction.
\end{proof}

\section{Concluding remarks}\label{sec:conclusion}
Let $f(g,\delta)$ denote the minimum zero forcing number over all graphs of girth $g$ and minimum
degree $\delta$. Theorem~\ref{thm:result} provides a lower bound for $f$, and
from~\cite{DavilaKenter} we know that this bound is tight in the following cases:
\begin{itemize}\itemsep=.2cm
\item $f(g,2)=2$ for all $g\geq 3$ (the $g$-cycle),
\item $f(3,\delta)=\delta$ for all $\delta\geq 1$ (the complete graph $K_{\delta+1}$),
\item $f(4,\delta)=2\delta-2$ for all $\delta\geq 2$ (the complete bipartite graph
  $K_{\delta,\delta}$),
\item $f(4,3)=4$ (the $3$-cube),
\item $f(5,3)=5$ (the Petersen graph),
\item $f(6,3)=6$ (the Heawood graph).
\end{itemize}
Consequently, the smallest open cases are the following.
\begin{question}
We know $7\leq f(7,3)\leq 8$ and $8\leq f(8,3)\leq 10$. Can we close these gaps?
\end{question}
\begin{question}
We know $f(5,4)\geq 8$. What is the best upper bound we can come up with?  
\end{question}
In general the bound $f(g,\delta)\geq\delta+(g-3)(\delta-2)$ is not sharp. For instance, using
essentially the same argument as in the proof of Theorem~\ref{thm:result}, one can prove
$f(g,\delta)\geq\delta+(g-3)(\delta-2)+1$ for $g\geq 14$, $\delta\geq 3$, and more generally, for
large values of $\delta$ and $g$ the exponential lower bound established in~\cite{Kalinowski2017} is
stronger than the bound from the present note. This motivates the following questions.
\begin{question}
What are upper bounds for $f(g,\delta)$? 
\end{question}
\begin{question}
  What can be said about the asymptotic behavior of $f(g,\delta)$?
\end{question}
 
\bibliographystyle{amsplain}
\bibliography{zero_forcing}

\providecommand{\bysame}{\leavevmode\hbox to3em{\hrulefill}\thinspace}
\providecommand{\MR}{\relax\ifhmode\unskip\space\fi MR }
\providecommand{\MRhref}[2]{%
  \href{http://www.ams.org/mathscinet-getitem?mr=#1}{#2}
}
\providecommand{\href}[2]{#2}
\begin{thebibliography}{10}

\bibitem{A08}
Ashkan Aazami, \emph{Hardness results and approximation algorithms for some
  problems on graphs}, Ph.D. thesis, University of Waterloo, 2008,
  \href{https://uwspace.uwaterloo.ca/handle/10012/4147}{uwspace.uwaterloo.ca/handle/10012/4147}.

\bibitem{AbajoDian2015}
Encarnaci\'on Abajo and Ana~Rosa Di{\'a}nez, \emph{Exact value of
  ex$(n;\{C_3,\dotsc, C_s\})$ for $n\leq\lfloor 25(s-1)/8\rfloor$}, Discrete
  Applied Mathematics \textbf{185} (2015), 1--7.

\bibitem{AIM-2008-Zeroforcingsets}
{AIM Minimum Rank -- Special Graphs Work Group}, \emph{Zero forcing sets and
  the minimum rank of graphs}, Linear Algebra and its Applications \textbf{428}
  (2008), no.~7, 1628--1648.

\bibitem{BarioliBarrettFallatEtAl2010}
Francesco Barioli, Wayne Barrett, Shaun~M. Fallat, H.~Tracy Hall, Leslie
  Hogben, Bryan Shader, Pauline Van Den~Driessche, and Hein Van Der~Holst,
  \emph{Zero forcing parameters and minimum rank problems}, Linear Algebra and
  its Applications \textbf{433} (2010), no.~2, 401--411.

\bibitem{BarioliBarrettFallatHallHogbenShaderDriesscheHolst-2012-ParametersRelatedto}
Francesco Barioli, Wayne Barrett, Shaun~M. Fallat, H.~Tracy Hall, Leslie
  Hogben, Bryan Shader, Pauline van~den Driessche, and Hein Van Der~Holst,
  \emph{Parameters related to tree-width, zero forcing, and maximum nullity of
  a graph}, Journal of Graph Theory \textbf{72} (2013), no.~2, 146--177.

\bibitem{Davila2017}
Randy Davila and Michael~A. Henning, \emph{The forcing number of graphs with
  given girth}, Quaestiones Mathematicae (2017), 1--16.

\bibitem{DavilaKenter}
Randy Davila and Franklin Kenter, \emph{Bounds for the zero forcing number of
  graphs with large girth}, Theory and Applications of Graphs \textbf{2}
  (2015), no.~2, Article 1.

\bibitem{Diestel2017}
Reinhard Diestel, \emph{Graph theory}, Graduate Texts in Mathematics, vol. 173,
  Springer, 2017.

\bibitem{EdholmHogbenHuynhEtAl2012}
Christina~J. Edholm, Leslie Hogben, Joshua LaGrange, and Darren~D. Row,
  \emph{Vertex and edge spread of zero forcing number, maximum nullity, and
  minimum rank of a graph}, Linear Algebra and its Applications \textbf{436}
  (2012), no.~12, 4352--4372.

\bibitem{Furst2017}
Maximilian F\"urst and Dieter Rautenbach, \emph{A short proof for a lower bound
  on the zero forcing number},
  arXiv:\href{https://arxiv.org/abs/1705.08365}{1705.08365}, 2017.

\bibitem{Genter1}
Michael Gentner, Lucia~D. Penso, Dieter Rautenbach, and U{\'e}verton~S. Souza,
  \emph{Extremal values and bounds for the zero forcing number}, Discrete
  Applied Mathematics \textbf{214} (2016), 196--200.

\bibitem{Gentner2018}
Michael Gentner and Dieter Rautenbach, \emph{Some bounds on the zero forcing
  number of a graph}, Discrete Applied Mathematics \textbf{236} (2018),
  203--213.

\bibitem{Kalinowski2017}
Thomas Kalinowski, Nina Kam\v{c}ev, and Benny Sudakov, \emph{Zero forcing
  number of graphs}, arXiv:\href{https://arxiv.org/abs/1705.10391}{1705.10391},
  2017.

\bibitem{Lu.etal_2015_Proofconjecturezero}
Leihao Lu, Baoyindureng Wu, and Zixing Tang, \emph{Proof of a conjecture on the
  zero forcing number of a graph}, Discrete Applied Mathematics \textbf{213}
  (2016), 233--237.

\bibitem{Mantel1907}
W.~Mantel, \emph{Problem 28 (solution by h. gouwentak, w. mantel, j. teixeira
  de mattes, f. schuh and w.a. wythoff)}, Wiskundige Opgaven \textbf{10}
  (1907), 60--61.

\bibitem{Turan1941}
Paul Tur\'an, \emph{On an extremal problem in graph theory}, Matematikai \'es
  Fizikai Lapok \textbf{48} (1941), 436--452, (in Hungarian).

\end{thebibliography}

\end{document}